\documentclass{article}
\usepackage{graphicx}
\usepackage[utf8]{inputenc}
\usepackage{subfigure}
\usepackage{amsmath}
\usepackage{amsthm}
\usepackage{amssymb, mathtools, hyperref}
\usepackage{graphicx}
\usepackage{cleveref}
\usepackage{tikz}
\usetikzlibrary{patterns,shapes}
\usepackage{gastex}
\usepackage{cite}
\usepackage{esvect} 

\def\RT{{\sf RT}}

\def\N{\mathbb N}
\def\A{\mathcal A}
\def\B{\mathcal B}

\def\R{\mathcal R}
\def\uu{\mathbf{u}}
\def\vvv{\mathbf{v}}

\theoremstyle{definition}
\newtheorem{definition}{Definition}
\newtheorem{corollary}[definition]{Corollary}
\newtheorem{remark}[definition]{Remark}
\newtheorem{example}[definition]{Example}

\theoremstyle{plain}
\newtheorem{theorem}[definition]{Theorem}
\newtheorem{proposition}[definition]{Proposition}
\newtheorem{lemma}[definition]{Lemma}

\title{The repetition threshold of episturmian sequences}

\author{L\!'ubom\'ira Dvo\v r\'akov\'a \footnote{Czech Technical University in Prague, Czech Republic.\\
e-mail \href{mailto:lubomira.dvorakova@fjfi.cvut.cz}{\tt lubomira.dvorakova@fjfi.cvut.cz}.}
\and
Edita Pelantov\'a \footnote{Czech Technical University in Prague, Czech Republic.\\ e-mail \href{mailto:edita.pelantova@fjfi.cvut.cz}{\tt edita.pelantova@fjfi.cvut.cz}.}}

\begin{document}

\maketitle

\begin{abstract} 
The repetition threshold of a class $C$ of infinite $d$-ary sequences is the smallest real number $r$  such that in the class $C$ there exists a~sequence that avoids $e$-powers for all $e> r$.
This notion  was introduced by  Dejean in 1972 for the class of all sequences over a $d$-letter alphabet. Thanks to the effort of many authors over more than 30 years, the precise value of the repetition threshold in this class is known for  every $d \in \N$. 
The repetition threshold for the class of Sturmian sequences was determined by Carpi and de Luca in 2000.
Sturmian sequences may be equivalently defined in various ways, therefore there exist many generalizations to larger alphabets.  
Rampersad, Shallit and Vandome in 2020 initiated a study of the repetition threshold for the class of balanced sequences -- one of the possible generalizations of Sturmian sequences. Here, we focus on the class of $d$-ary episturmian sequences -- another generalization of Sturmian sequences introduced by Droubay, Justin and Pirillo in 2001.  We show that the repetition threshold of this class is reached by the $d$-bonacci sequence and its value equals $2+\frac{1}{t-1}$, where $t>1$ is the unique positive root of the polynomial $x^d-x^{d-1}-\cdots -x-1$.
\end{abstract}

\section{Introduction}\label{sec:Introduction}

In combinatorics on words, the topics that have been attracting interest of researchers from the very beginning up to nowadays are {\em repetition, critical exponent} and {\em repetition threshold}. A word $v$ is an $e$-{\em power} of a word $u$ if $v$ is a~prefix of the infinite periodic sequence $uuu\cdots = u^\omega$ and  $e=|v|/|u|$, where $|v|$, resp. $|u|$ denotes the length of $v$, resp. $u$. We write $v=u^e$.  For instance, a Czech word $kapka$ (drop) can be written in this formalism  as  $(kap)^{5/3}$.

The {\em critical exponent}  $E(\uu )$ of an infinite sequence $\uu$ is defined as
$$E(\uu) =\sup\{e \in \mathbb{Q}: \  u^e \  \text{is a factor of   } \uu  \  \text{for a non-empty word} \  u\}
$$
and the {\em repetition threshold} as
$$ \RT(d) = \inf\{ E(\uu) : \uu \text{\ is a $d$-ary sequence} \}\,.
$$ 
The repetition threshold may be interpreted in the following way: $\RT(d)$ is the smallest real number such that there exists a $d$-ary sequence that avoids  $e$-powers  for every exponent $e > RT(d)$. 

The result by Axel Thue from 1912 stating that $\RT(2) = 2$ is considered to be the origin of combinatorics on words. 
Dejean \cite{Dej72} showed that $\RT(3)=7/4$ and conjectured the remaining values $\RT(4)=7/5$ (proved by Pansiot \cite{Pan84c}) and $\RT(d)=1+\frac{1}{d-1}$ for $d\geq 5$ (proved by efforts of many authors \cite{Mou92,Car07,CuRa11,Rao11}). 

Cassaigne in \cite{Ca2008} introduces  also an asymptotic version of the critical exponent  in the following way:
The {\em asymptotic critical exponent} $E^*(\uu)$  is defined to be $\infty$  if $E(\uu) = \infty$ and 
$$E^*(\uu) =\lim_{n\to \infty}\sup\{e \in \mathbb{Q}: \  u ^e \  \text{is a factor of  } \uu  \  \text{for some }  u \ \text{of length} \geq  n  \}\,,$$
 otherwise.~\footnote{Let us emphasize that the terminology is not unique: the asymptotic critical exponent is called asymptotic index in~\cite{Ca2008}, and it is even called critical exponent in~\cite{Van2000, JuPi2002, GlJu2009}.} 
 We define the {\em asymptotic repetition threshold} as $$\RT^*(d) =\inf\{E^*(\uu): \uu \text{\ is a $d$-ary sequence}\}\,. $$
In contrast to the repetition threshold, for the asymptotic repetition threshold we have $\RT^*(d) =1 \ \text{for all $d\geq 2$}\,,$ see~\cite{Ca2008}.

An important parameter of any class $C$ of sequences is the infimum of critical exponents of sequences from this class.
For the class of all $d$-ary sequences this number equals $\RT(d)$ by definition; hence we write $\RT(C)$ to denote this number and refer to it as the \emph{repetition threshold for the class $C$}.
Let us briefly survey what is known about $\RT(C)$ for different classes $C$.

The most studied aperiodic binary sequences are Sturmian sequences. The most prominent example is the Fibonacci sequence, which may be obtained by a repeated application of the rewriting rule $\tt 0 \mapsto \tt 01$  and $\tt 1\mapsto \tt 0$ on the starting letter $\tt 0$,
$$
\tt 0 \mapsto \tt 01 \mapsto \tt 010\mapsto \tt 01001 \mapsto \tt 01001010 \mapsto \tt 0100101001001 \cdots 
$$
The Fibonacci sequence has the privileged position among all Sturmian sequences, see~\cite{Ca2008}. In particular, in the class of Sturmian sequences, the Fibonacci sequence has the minimal value of both the critical and the asymptotic critical exponent. As shown by Carpi and de Luca~\cite{CaDeLu00} for the repetition threshold and Vandeth~\cite{Van2000} for the asymptotic repetition threshold, in the class of Sturmian sequences, $\RT(C)=\RT^*(C)=2+\frac{1+\sqrt{5}}{2}$.

Recently, several groups of researchers have focused on the repetition threshold of some special classes: balanced sequences, episturmian sequences and  sequences  rich in palindromes. All these classes may be understood as a generalization of Sturmian sequences since each Sturmian sequence is balanced, rich in palindromes and episturmian.     
Let us underline that the class of episturmian sequences is a subclass of the class of sequences rich in palindromes.

Rampersad, Shallit and Vandomme~\cite{RSV19} suggested to study the repetition threshold for the class $C_d$ of $d$-ary balanced sequences. The following results have been proved so far:
\begin{itemize}
    \item $\RT(C_2)=2+\frac{1+\sqrt{5}}{2}$ \cite{CaDeLu00};
    \item $\RT(C_3)=2+\frac{\sqrt{2}}{2}$ and $\RT(C_4)=1+\frac{1+\sqrt{5}}{4}$ \cite{RSV19};
    \item $\RT(C_d)=1+\frac{1}{d-3}$ for $5 \leq d\leq 10$ \cite{Bar20, BaSh19, DDP21};
    \item $\RT(C_d)=1+\frac{1}{d-2}$ for $d=11$ and all even numbers $d \geq 12$ \cite{DvOpPeSh2022}.
    \end{itemize}
It remains as an open problem to prove the conjecture $\RT(C_d)=1+\frac{1}{d-2}$ also for all odd numbers $d\geq 13$.

For the asymptotic repetition threshold in this class, the situation is more interesting than for general $d$-ary sequences, see~\cite{DOP2022} and \cite{DvPe2022}.   
\begin{itemize}
    \item $\RT(C_d)= \RT^*(C_d)$ for $d =  2,3,4,5$;
 \item $\RT(C_d)> \RT^*(C_d)$ for $d\geq 6$; 
    \item $1+ \frac{1}{2^{d-2}}< \RT^*(C_d)< 1+ \frac{\tau^3}{2^{d-3}}$, where $\tau = \frac{{1}+\sqrt{5}}{2}$;
    \item the precise value of $\RT^*(C_d)$ is known only for $d \leq 10$. 
\end{itemize}
Sequences rich in palindromes have not been sufficiently studied and characterized so far. Despite of this fact, Currie, Mol and Rampersad~\cite{CuMoRa2020} determined the values of the repetition threshold and the asymptotic repetition threshold for the binary alphabet.
If $C_2 = $ the set of all binary sequences rich in palindromes, then 
$$ \RT(C_2)= \RT^*(C_2) =2+\frac{\sqrt{2}}{2} \sim 2.707\,.$$
Moreover, Baranwal and Shallit~\cite{BaSh19+} found a lower bound $\RT(C_3)\geq  9/4$ for the class $C_3$ of ternary sequences rich in palindromes.

In this paper, we will study the class of $d$-ary episturmian sequences.  They were introduced by  Droubay, Justin and Pirillo~\cite{DrJuPi2001}.  For $d=2$, this class coincides with the class of binary balanced sequences, hence the repetition threshold is well known.  
The paper~\cite{DL23} is devoted to a particular subclass of episturmian sequences -- regular Arnoux-Rauzy sequences. It is proved there that the $d$-bonacci sequence $\uu_d$ has the minimal critical and asymptotic critical exponent among all regular $d$-ary Arnoux-Rauzy sequences. The value of $E^*(\uu_d)$ is derived, however it is only conjectured ibidem that $E(\uu_d)=E^*(\uu_d)$. Previously, Glen and Justin~\cite{GlJu2009} determined the value of $E^*(\uu_d)$ using a formula from Justin and Pirillo~\cite{JuPi2002}. We have to point out that they call {\em critical exponent} what is called here {\em asymptotic critical exponent}!

Let us recall that $\uu_d$ may be obtained by a repeated application of the rewriting rules
$\tt 0 \mapsto \tt 01$,  \  $\tt 1\mapsto \tt 02$, \    $\tt 2\mapsto \tt 03, \  \ldots\ , \ \tt d-2 \mapsto \tt 0(d-1)$ 
\ and  \ $\tt d-1 \mapsto \tt 0$.   

\noindent In this paper, we will derive that if $C_d = $ the set of all $d$-ary  episturmian  sequences, 
then, for every $d \in \N, \ d\geq 2$,
$$ \RT(C_d)= \RT^*(C_d) =  2+\frac{1}{t-1}, 
$$
where $t>1$ is the only positive root of the polynomial $x^d-x^{d-1}-\dots -x-1$.  In the proof, we will use the fact that the class $C_d$ is an $S$-adic system, see \cite{BeDe14},  where the set $S$ contains $d$ morphisms. Each morphism is represented by its incidence matrix. We transform the task to determine the asymptotic repetition threshold to a task related to the joint spectral radius of the set of $d$ incidence matrices.       

\medskip 

The paper is organized as follows. After Section 2 containing preliminaries, we derive in Section 3 a matrix formula for the critical exponent and the asymptotic critical exponent of Arnoux-Rauzy sequences. In Section 4 we provide the value of the asymptotic critical exponent of the $d$-bonacci sequence and prove that the critical exponent has the same value. In Section 5 we show that the $d$-bonacci sequence has the longest bispecial factors and we determine the asymptotic behaviour of their lengths. It enables us in Section 6  to derive  the repetition threshold  and asymptotic repetition threshold for the class of $d$-ary episturmian sequences.

\section{Preliminaries}
\subsection{Basic notions}

An \textit{alphabet} $\mathcal A$ is a finite set and its elements are called \textit{letters}. Throughout this paper, we use ${\mathcal A}=\{\tt 0,1,\dots, d-1\}$.
A \textit{word} $u$ over $\mathcal A$ of \textit{length} $n$ is a finite string $u = u_0 u_1 \cdots u_{n-1}$, where $u_j\in\mathcal A$ for all $j \in \{0,1,\dots, n-1\}$. The length of $u$ is denoted $|u|$ and $|u|_{\tt i}$ denotes the number of occurrences of the letter ${\tt i}\in\mathcal A$ in the word $u$. The \textit{Parikh vector} $ \vec{u} \in \N^{d}$ is the vector defined as ${\vec u } = (|u|_{\tt 0}, |u|_{\tt 1}, \dots, |u|_{\tt d-1})^{ T}$.
The set of all finite words over $\A$ is denoted $\A^*$. The set $\A^*$ equipped with concatenation as the operation forms a monoid with the \textit{empty word} $\varepsilon$ as the neutral element. 

A \textit{sequence} $\uu$ over $\A$ is an infinite string $\uu = u_0 u_1 u_2 \cdots$ of letters $u_j \in \A$ for all $j \in \N$.  A \textit{factor} of a sequence $\uu = u_0 u_1 u_2 \cdots$ is a word $w\in\A^*$ such that $w = u_j u_{j+1} u_{j+2} \cdots u_{\ell-1}$ for some $j, \ell \in \N$, $j \leq \ell$. The integer $j$ is called an \textit{occurrence} of the factor $w$ in the sequence $\uu$. If $j=0$, then $w$ is a \textit{prefix} of $\uu$.

The \textit{language} $\mathcal{L}(\uu)$ of a sequence $\uu$ is the set of factors occurring in $\uu$.
The language $\mathcal{L}(\uu)$ is called \textit{closed under reversal} if for each factor $w=w_0w_1\cdots w_{n-1}$, its mirror image $w_{n-1}\cdots w_1 w_0$ is also a factor of $\uu$.
A~factor $w$ of a sequence $\uu$ is \textit{right special} if $w{\tt i}, w{\tt j} \in \mathcal{L}(\uu)$ for at least two distinct letters ${\tt i, j} \in \A$. A \textit{left special} factor is defined analogously.
A factor is called \textit{bispecial} if it is both left and right special. 

A sequence $\uu$ is \textit{recurrent} if each factor of $\uu$ has at least two occurrences in $\uu$. Moreover, a recurrent sequence $\uu$ is \textit{uniformly recurrent} if the distances between the consecutive occurrences of each factor in $\uu$ are bounded.
A sequence $\uu$ is \textit{eventually periodic} if there exist words $w \in \A^*$ and $v \in \A^* \setminus \{\varepsilon\}$ such that $\uu$ can be written as $\uu = wvvv \cdots = wv^\omega$. In particular, $\uu$ is \textit{periodic} if $w=\varepsilon$. If $\uu$ is not eventually periodic, $\uu$ is called \textit{aperiodic}.

Consider a factor $w$ of a recurrent sequence $\uu = u_0 u_1 u_2 \cdots$. Let $j < \ell$ be two consecutive occurrences of $w$ in $\uu$. Then the word $u_j u_{j+1} \cdots u_{\ell-1}$ is a \textit{return word} to $w$ in $\uu$.
The set of all return words to $w$ in $\uu$ is denoted by $\R_\uu(w)$. The set $\R_\uu(w)$ is finite if and only if $\uu$ is uniformly recurrent.
A \textit{morphism} is a map $\psi: \A^* \to \B^*$ such that $\psi(uv) = \psi(u)\psi(v)$  for all words $u, v \in \A^*$.
The morphism $\psi$ is called \textit{non-erasing} if $\psi(\tt i)\not =\varepsilon$ for each ${\tt i} \in \A$.
Morphisms can be naturally extended to sequences by setting
$\psi(u_0 u_1 u_2 \cdots) = \psi(u_0) \psi(u_1) \psi(u_2) \cdots\,$.
A \textit{fixed point} of a morphism $\psi:  \A^* \to  \A^*$ is a sequence $\uu$ such that $\psi(\uu) = \uu$.
We associate to a morphism $\psi: \A^* \to  \A^*$ the \textit{(incidence) matrix} $M_\psi$ defined for each $k,j \in \{0,1,\dots, d-1\}$ as $[M_\psi]_{kj}=|\psi(\tt j)|_{\tt k}$. 

By definition, we have for each $u \in \A^*$ the following relation for the Parikh vectors $\vv{\psi(u)}=M_\psi\vec{u}$.
Similarly as for matrices, we index the components of vectors from zero. We denote $\vec e_i$ the $i$-th vector of the canonical basis of $\mathbb R^d$ for $i \in \{0,1,\dots, d-1\}$, i.e., the $i$-th component of $\vec e_i$ equals one, while all other components are zero. $I$ denotes the identity matrix and ${\vec 1 }=(1,1,\dots, 1)^T \in \mathbb N^d$. To any permutation $\pi$ on $\{0,1,\dots, d-1\}$ we associate the permutation matrix $P \in \mathbb N^d \times \mathbb N^d$ defined by $\vec e_i^{\ T}P=\vec e_{\pi(i)}^{\ T}$.

\subsection{Critical exponent}
The (asymptotic) critical exponent and the (asymptotic) repetition threshold have been defined already in Introduction. 
In this section, we recall first a useful formula from~\cite{DDP21} for the computation of the critical exponent and the asymptotic critical exponent of uniformly recurrent aperiodic sequences. It is based on the lengths of bispecial factors and their shortest return words.

\begin{theorem}[\rm\cite{DDP21}]
\label{thm:FormulaForE}
Let $\uu$ be a uniformly recurrent aperiodic sequence.
Let $(b_n)_{n\in\N}$ be the sequence of all bispecial factors in $\uu$ ordered by length.
For every $n \in \N$, let $r_n$ be the shortest return word to the bispecial factor $b_n$ in $\uu$.
Then
$$
E(\uu) = 1 + \sup\limits_{n \in \N} \left\{ \frac{|b_n|}{|r_n|} \right\}
\qquad \text{and} \qquad
E^*(\uu) = 1 + \limsup\limits_{n \to \infty}  \frac{|b_n|}{|r_n|} .
$$
\end{theorem}

Second, we derive a relation between the asymptotic critical exponents of a~sequence and its morphic image. In order to do that we have to recall the definition of uniform letter frequencies.
Let $\uu$ be a~sequence over an alphabet $\mathcal A$ and $\tt i$ a letter in $\mathcal A$. Then the {\em uniform frequency} $f_{\tt i}$ of the letter $\tt i$ is equal to $\alpha$ if for any sequence $(w^{(n)})$ of factors of $\uu$ with increasing lengths 
$$\alpha=\lim_{n\to \infty}\frac{|w^{(n)}|_{\tt i}}{|w^{(n)}|}\,.$$

\begin{lemma}\label{lem:morphismE}
Let $\vvv$ be a sequence over an alphabet $\mathcal A$ such that the uniform frequency of every letter exists. Then $E^*(\psi(\vvv))\geq E^*(\vvv)$ for any non-erasing morphism $\psi:{\mathcal A}^* \mapsto {\mathcal B}^*$.
\end{lemma}
\begin{proof}
According to the definition of $E^*(\vvv)$, there exist sequences $\bigl(w^{(n)}\bigr)$ and $\bigl(v^{(n)}\bigr)$ such that 
\begin{enumerate}
\item $\lim\limits_{n\to \infty} |v^{(n)}| = \infty$; 
\item $w^{(n)}$ is a factor of $\vvv$ for each $n \in \N$; 
\item $w^{(n)}$ is a prefix of the periodic sequence $\bigl(v^{(n)}\bigr)^\omega$ for each $n \in \N$; 
\item $E^*(\vvv)=\lim\limits_{n\to \infty}\frac{|w^{(n)}|}{|v^{(n)}|}$.
\end{enumerate}

Obviously, $\psi(w^{(n)})$ is a factor of $\psi(\vvv)$ and $\psi(w^{(n)})$ is a prefix of  $\psi\bigl(\bigl(v^{(n)}\bigr)^\omega\bigr) = \bigl(\psi(v^{(n)})\bigr)^\omega$ for each $n \in \N$. As the morphism $\psi$ is non-erasing, 
$\lim |\psi(v^{(n)})| = \infty$. Consequently, 
\begin{equation}\label{eq:pravastrana}
 E^*(\psi(\vvv))\geq \lim_{n \to \infty}\frac{|\psi(w^{(n)})|}{|\psi(v^{(n)})|}.
\end{equation}
 To complete the proof, it suffices to show that the limit on the right-hand side is equal to $E^*(\vvv)$. Combining two simple facts: 
$\frac{|\psi(u)|}{|u|}={\vec 1 \hspace{0.01cm}}^TM_\psi \frac{\vec u}{|u|}$ for each word $u$ over $\mathcal A$ and $\lim\limits_{n\to \infty}\frac{\vec w^{(n)}}{|w^{(n)}|} = \lim\limits_{n\to \infty}\frac{\vec v^{(n)}}{|v^{(n)}|} = \vec f\,$, where $\vec f$ is the vector of the letter frequencies in $\vvv$, we obtain
$$
\lim_{n \to \infty}\frac{|\psi(w^{(n)})|}{|w^{(n)}|} =  {\vec 1 \hspace{0.01cm}}^TM_\psi \vec f   = \lim_{n \to \infty}\frac{|\psi(v^{(n)})|}{|v^{(n)}|} \,.
$$
Consequently, 
$$
\lim_{n \to \infty} \frac{|\psi(w^{(n)})|}{|\psi(v^{(n)})|} =\lim_{n \to \infty}  \frac{|\psi(w^{(n)})|}{|w^{(n)}|} \frac{|v^{(n)}|}{|\psi(v^{(n)})|} \frac{|w^{(n)}|}{|v^{(n)}|} \ =  \  E^*(\vvv)\,.$$
\end{proof}

\subsection{Episturmian sequences}

A sequence $\uu$ over $\mathcal A$ is {\em episturmian} if the language ${\mathcal L}(\uu)$ is closed under reversal and if there exists at most one right special factor $w$ of each length. An episturmian sequence $\uu$ is {\em standard} if all left special factors of $\uu$ are prefixes of $\uu$. For each episturmian sequence there exists a unique standard episturmian sequence with the same language.
If the sequence $\uu$ is episturmian and each right special factor $w$ has all possible right extensions, i.e., $w{\tt i} \in {\mathcal L}(\uu)$ for each letter ${\tt i} \in {\mathcal A}$, then $\uu$ is called {\em Arnoux-Rauzy} (AR for short). 
Binary AR sequences are called {\em Sturmian} sequences.
Episturmian sequences are uniformly recurrent~\cite{DrJuPi2001}. 

We use in the sequel the S-adic representation of $d$-ary episturmian sequences based on the following morphisms.
For every $\tt i\in\{\tt 0, \tt 1, \dots, \tt d-1\}$, let $\varphi_{\tt i}$ denote the morphism
$\varphi_{{\tt i}}\colon \{\tt 0,\tt 1 \dots, \tt d-1\}^*\to \{\tt 0,\tt 1 \dots, \tt d-1\}^*$ such that $\varphi_{\tt i}\colon \tt i \mapsto \tt i$ and $\varphi_{\tt i}\colon \tt j\mapsto \tt ij \quad \text{for every $\tt j\not =\tt i$.}$
The matrix $M_{\varphi_{\tt i}}$ of the morphism $\varphi_{{\tt i}}$ is thus defined, for $k,j \in \{0, 1,\dots, d-1\}$, \begin{equation}\label{eq:elementary_morphisms}
[M_{\varphi_{\tt i}}]_{kj}=\left\{ \begin{array}{rcl}
&1 &\text{for $k=i$};\\
&1 &\text{for $k=j$};\\
&0 &\text{otherwise}.\\
\end{array}\right.
\end{equation}
\begin{example}
To get a better idea, let us illustrate $M_{\varphi_{\tt i}}$ for $d=4$ and $i=2$:
$$M_{\varphi_{\tt 2}}=
\left(\begin{array}{cccc}
1&0&0&0\\
0&1&0&0\\
1&1&1&1\\
0&0&0&1
\end{array}\right)\,.$$
\end{example}
\begin{theorem}[\cite{DrJuPi2001}]\label{thm:directiveAR}
For each standard episturmian sequence over $\{\tt 0,\tt 1,\dots, \tt d-1\}$ there exists a unique sequence of morphisms $\Delta=(\psi_{n})_{n=1}^{\infty}$, where $\psi_n \in \{\varphi_{{\tt 0}},\varphi_{{\tt 1}},\dots, \varphi_{{\tt d-1}}\}$, and a unique sequence of standard episturmian sequences $(\uu^{(n)})_{n=1}^{\infty}$ such that 
$$\uu=\psi_{1}\psi_{2}\psi_{3}\cdots \psi_{n}(\uu^{(n)}), \text{ for every } n\geq 1.$$
\end{theorem}
The sequence of morphisms $\Delta$ is called the {\em directive sequence} of $\uu$. 
We associate the same directive sequence to each episturmian sequence having the same language.
On one hand, it is well-known that a $d$-ary episturmian sequence is AR if and only if every morphism from the set $\{\varphi_{{\tt 0}},\varphi_{{\tt 1}},\dots, \varphi_{{\tt d-1}}\}$ occurs in the directive sequence infinitely many times.
On the other hand, if an episturmian sequence is eventually periodic, then it is periodic and its directive sequence is eventually constant. It follows immediately that for each aperiodic $d$-ary standard episturmian sequence $\uu$ that is not AR, there exist $\psi_1, \dots, \psi_n \in \{\varphi_{{\tt 0}},\varphi_{{\tt 1}},\dots, \varphi_{{\tt d-1}}\}$ such that $\uu=\psi_1\cdots \psi_n(\vvv)$ for some $d'$-ary AR sequence, where $2\leq d' <d$.

If there exists a sequence of positive integers $(a_n)_{n \in \N}$ such that the directive sequence $\Delta$ has the form $$\Delta = \varphi^{a_0}_{\tt 0} \varphi^{a_1}_{\tt 1} \cdots \varphi_{\tt d-1}^{a_{d-1}}\varphi^{a_d}_{\tt 0} \varphi^{a_{d+1}}_{\tt 1} \cdots \varphi_{\tt d-1}^{a_{2d-1} }\varphi_{\tt 0}^{a_{2d}} \cdots\,,  $$
then $\uu$ is called a \textit{regular} AR sequence (the name was introduced by Peltom\"{a}ki~\cite{Pel21}).

\begin{example}\label{ex:d-bonacci}
    The most prominent  $d$-ary AR sequence is the $d$-bonacci sequence $\uu_d$ having the directive sequence $\Delta=(\varphi_{\tt 0}\varphi_{\tt 1}\cdots \varphi_{\tt d-1})^{\omega}$. In particular, $\uu_2$ is the {\em Fibonacci sequence} and $\uu_3$ the {\em Tribonacci sequence}.  Obviously, $\uu_d$ is a regular AR sequence. 
\end{example}

\section{Critical exponent of AR sequences}
In this section, we will derive a matrix formula for the critical exponent and the asymptotic critical exponent of AR sequences.
For regular AR sequences, a   formula based on the $S$-adic representation is given  in~\cite{DL23}. 
Previously, Justin and Pirillo~\cite{JuPi2002} provided a formula for the asymptotic critical exponent of AR sequences fixed by a~morphism. Integer powers in regular AR sequences were studied by Glen~\cite{Glen07}.

\begin{lemma}[\cite{DrJuPi2001, Glen07}]\label{lem:bispecials_d}
Let $\tt i\in\{\tt{0}, \tt{1}, \dots, \tt{d-1}\}$ and let 
$\uu, \vvv$ be standard $d$-ary episturmian sequences. Let
$r \in {\mathcal L}(\vvv)$.
If $\uu=\varphi_{\tt i}(\vvv)$, then
\begin{enumerate}
    \item if $b$ is a bispecial factor in $\vvv$, then  $\varphi_{\tt i}(b){\tt i}$ is a bispecial factor in $\uu$;
    \item if $w$ is a non-empty bispecial factor in $\uu$, then $w=\varphi_{\tt i}(b){\tt i}$ for some bispecial factor $b$ in $\vvv$;
    \item $r\in\mathcal{R}_{\vvv}(b)$  if and only if $\varphi_{\tt i}(r)\in\mathcal{R}_\uu(\varphi_{\tt i}(b){\tt i})$.   
\end{enumerate}
\end{lemma}
\begin{corollary}\label{coro:r_N_and_b_N} Let $\uu$ be an AR sequence.
 Denote $b_N$ the $N$-th bispecial factor of $\uu$ (when ordered by length) and $r_N$ the shortest return word to $b_N$. Then $b_N=r_N b_{N-1}$ for $N \geq 1$.
 \end{corollary}

 \begin{proof}
Let us prove the statement by induction.
Let $\varphi_{\tt i}$ be the first element of the directive sequence of $\uu$. Then $\uu=\varphi_{\tt i}(\vvv)$, where $\vvv$ is an AR sequence. 
Evidently, the empty word $\varepsilon$ is a bispecial factor in $\vvv$ and $\{\tt 0, 1, \dots, d-1\}$ is the set of its return words. By Lemma~\ref{lem:bispecials_d}, $b_1=\varphi_{\tt i}(\varepsilon)\tt i=\tt i$ and $r_1=\varphi_{\tt i}(\tt i)=\tt i$, thus $b_1=r_1 b_0={\tt i}\varepsilon$. 
 
Assume $b_N=r_N b_{N-1}$ holds for some $N \geq 1$ not only in the sequence $\uu$, but in any AR sequence. Then by Lemma~\ref{lem:bispecials_d} and using the induction assumption, $b_{N+1}=\varphi_{\tt i}(b_N^{\vvv}){\tt i}=\varphi_{\tt i}(r_N^{\vvv}b_{N-1}^{\vvv}){\tt i}=\varphi_{\tt i}(r_N^{\vvv})\varphi_{\tt i}(b_{N-1}^{\vvv}){\tt i}=r_{N+1}b_N$,
where $b_N^{\vvv}$, resp. $r_N^{\vvv}$ denotes the $N$-th bispecial factor in $\vvv$, resp. its shortest return word.

 \end{proof}

\begin{lemma}\label{lem:BS_shortest_retwords}
Let $\uu$ be an AR sequence with the directive sequence $\Delta=(\psi_{n})_{n=1}^{\infty}$. 
Denote $b_N$ the $N$-th bispecial factor of $\uu$ (when ordered by length) and $r_N$ the shortest return word to $b_N$. Then for $N\geq 1$
$$|b_N|=\frac{1}{d-1}{\vec 1\hspace{0.01cm} }^T (M_{\psi_1} M_{\psi_2}\cdots M_{\psi_N}-I)\vec 1$$
and
$$|r_N|=\frac{1}{d-1}{\vec 1\hspace{0.01cm} }^{T} M_{\psi_1} M_{\psi_2}\cdots (M_{\psi_N}-I) \vec 1.$$
\end{lemma}
\begin{proof} 
First, let us deal with the lengths of return words, second, with the lengths of bispecial factors. 

Using Lemma~\ref{lem:bispecials_d}, each return word to $b_N$ is of the form $\psi_1\cdots \psi_N(\tt j)$, where $\tt j \in \{\tt 0, 1, \dots, d-1\}$.
If $\psi_N=\varphi_{\tt i}$, then $\varphi_{\tt i}(\tt i)=\tt i$ is a proper prefix of $\varphi_{\tt i}(\tt j)={\tt i}\tt j$ for each $\tt j \not =\tt i$. Consequently, the shortest return word is of the form
$r_N=\psi_1\cdots \psi_N(\tt i)$.
The length of $r_N$ satisfies
$$\begin{array}{rcl}
|r_N|&=&{\vec 1\hspace{0.01cm} }^T M_{\psi_1}\cdots M_{\psi_N}\vec e_{i}\\
&=&{\vec 1\hspace{0.01cm} }^T M_{\psi_1}\cdots M_{\psi_{N-1}}\vec e_{i}\\
&=&\frac{1}{d-1}{\vec 1\hspace{0.01cm} }^T M_{\psi_1}\cdots  M_{\psi_{N-1}}(M_{\psi_N}-I)\vec 1\,.
\end{array}$$

Now, we will prove by induction that $|b_N|=\frac{1}{d-1}{\vec 1\hspace{0.01cm} }^T (M_{\psi_1} M_{\psi_2}\cdots M_{\psi_N}-I)\vec 1$.
Let $\psi_1=\varphi_{\tt i}$, then by Lemma~\ref{lem:bispecials_d} we have $b_1={\tt i}$.
Thus, $|b_1|=\frac{1}{d-1}{\vec 1\hspace{0.01cm} }^T (M_{\psi_1}-I)\vec 1=\frac{1}{d-1}{\vec 1\hspace{0.01cm} }^T (d-1)\vec e_{i}=1$.
Assume $|b_N|=\frac{1}{d-1}{\vec 1\hspace{0.01cm} }^T (M_{\psi_1} M_{\psi_2}\cdots M_{\psi_N}-I)\vec 1$ holds for some $N \geq 1$. Then using Corollary~\ref{coro:r_N_and_b_N} and the already proved first statement, we have 
$$\begin{array}{rcl}
|b_{N+1}|&=&|r_{N+1}|+|b_{N}|\\
&=&\frac{1}{d-1}{\vec 1\hspace{0.01cm} }^{T} M_{\psi_1} \cdots M_{\psi_N}(M_{\psi_{N+1}}-I) \vec 1+\frac{1}{d-1}{\vec 1\hspace{0.01cm} }^T (M_{\psi_1} \cdots M_{\psi_N}-I)\vec 1\\
&=&\frac{1}{d-1}{\vec 1\hspace{0.01cm} }^{T} (M_{\psi_1} \cdots M_{\psi_N}M_{\psi_{N+1}}-I)\vec 1\,.
\end{array}$$
\end{proof}

Using Theorem~\ref{thm:FormulaForE} and Lemma~\ref{lem:BS_shortest_retwords}, we provide a matrix formula for the  critical and the asymptotic critical exponent. 
\begin{theorem}\label{thm:formulaE*andE}
Let $\uu$ be a $d$-ary  AR sequence with the directive sequence $\Delta=(\psi_{n})_{n=1}^{\infty}$. 
For all $N \in \mathbb N$, denote   $$s_N={\vec 1\hspace{0.01cm}}^T M_{\psi_1} M_{\psi_2}\cdots M_{\psi_N}\vec 1\,.$$  
Then $$E(\uu)=1+\sup\left\{\frac{s_N-d}{s_N-s_{N-1}}\right\} \quad \text{and} \quad E^*(\uu)=1+\limsup_{N \to \infty}\frac{s_N}{s_N-s_{N-1}} \,.$$
\end{theorem}

\section{Critical exponent of the $d$-bonacci sequence}\label{sec:d-bonacci}
The aim of this section is to prove that for the $d$-bonacci sequence $\uu_d$, it holds $E(\uu_d)=E^*(\uu_d)$. This was stated as a conjecture for $d\geq 4$ in~\cite{DL23} and proved for $d=2$ in~\cite{MiPi1992} and $d=3$ in~\cite{BoTan07}. The value $E^*(\uu_d)=2+\frac{1}{t-1}$, where $t>1$ is the unique positive root of the polynomial $x^d-x^{d-1}-\cdots -x-1$, was also determined in~\cite{GlJu2009, DL23}.

It is a well-known fact that the $d$-bonacci sequence $\uu_d$, defined in Example~\ref{ex:d-bonacci}, is the fixed point $\uu_d=\varphi(\uu_d)$ of the morphism 
$$\begin{array}{rccc} 
\varphi: 
&\tt 0 & \to & \tt 01 \\
&\tt 1 & \to & \tt 02 \\
& \tt 2 & \to & \tt 03 \\
& \vdots && \vdots \\
& \tt d-2 & \to & \tt 0(d-1) \\
& \tt d-1 & \to & \tt 0\,.
\end{array}$$

Let us denote $M$ the matrix of the morphism $\varphi$, i.e., for $k,j \in \{0, 1,\dots, d-1\}$ $$[M]_{kj}=\left\{ \begin{array}{rcl}
&1 &\text{if $k=0$};\\
&1 &\text{if $k=j+1$};\\
&0 &\text{otherwise}.\\
\end{array}\right.$$

\begin{example}
For a better idea, let us illustrate $M$ for $d=4$:
$$M=
\left(\begin{array}{cccc}
1&1&1&1\\
1&0&0&0\\
0&1&0&0\\
0&0&1&0
\end{array}\right)\,.$$
\end{example}

In the sequel, we will write $M_i$ instead of $M_{\varphi_{\tt i}}$ for the morphisms from~\eqref{eq:elementary_morphisms}. Moreover, $M_n$ for $n\in \mathbb N$ means $M_i$, where $i=n\hspace{-0.1cm}\mod d$.
\medskip

In order to compute $E^*(\uu_d)$ and $E(\uu_d)$, we will use the following lemmata. 
\begin{lemma}\label{lem:s_n_simplified}
Let $\uu_d$ be the $d$-bonacci sequence.
Then 
$$s_N={\vec 1\hspace{0.01cm}}^T M_{{{0}}} M_{{{1}}}\cdots M_{{{{N-1}}}}\vec 1={\vec 1\hspace{0.01cm}}^T M^{N}\vec 1$$ for all $N \in \mathbb N, N \geq 1$, and $s_0=d$.
\end{lemma}
\begin{proof}
We will need the following three observations. Let $R$ denote a $d\times d$ permutation matrix defined by $R=(\vec e_1 \dots \vec e_{d-1} \ \vec e_0)$.
\begin{enumerate}
\item $RR^T=I$ and $R^T\vec 1=\vec 1$;
\item $R^T M_{{i+1}}R=M_{i}$ for $i \in \mathbb N$;
\item $M_{{{0}}} M_{{{1}}}\cdots M_{{{{N-1}}}}R^{N}=M^{N}$ for all $N \in \mathbb N, N \geq 1$. \end{enumerate}
Let us prove by induction the third claim using the first two observations, which are simple to verify: $M_{{{0}}}R=M^1$ by definition of the matrices.
Assume $M_{{{0}}} M_{{{1}}}\cdots M_{{{{N-1}}}}R^{N}=M^{N}$ for some $N\geq 1$. Then 
$M_{{{0}}} M_{{{1}}}\cdots M_{{{{N}}}}R^{N+1}=M^{N}(R^T)^NM_{{{{N}}}}R^{N+1}=M^N M_0 R=M^{N+1}$.

Using the above observations, we may write
$$\begin{array}{rcl}
s_N&=&{\vec 1\hspace{0.01cm}}^T M_{{{0}}} M_{{{1}}}\cdots M_{{{{N-1}}}}\vec 1\\
&=&{\vec 1\hspace{0.01cm}}^T M_{{{0}}} M_{{{1}}}\cdots M_{{{{N-1}}}}R^N (R^T)^N\vec 1\\
&=&{\vec 1\hspace{0.01cm}}^T M^{N}\vec 1\,.
\end{array}$$
\end{proof}

\begin{lemma}\label{lem:recurrence_s_n}
Let $\uu_d$ be the $d$-bonacci sequence. 
Then $(s_N)$ satisfies the $d$-bonacci recurrence relation, i.e.,
$$s_N=s_{N-1}+\dots + s_{N-d} \quad \text{for $N \geq d$}$$
and $s_N=(d-1)2^N+1$ for $N \in \{0,1,\dots, d-1\}$.
\end{lemma}
\begin{proof}
First, let us check the initial conditions.
Clearly, $s_0=d=(d-1)2^0+1$. It is easy to prove for $1\leq N \leq d-1$
$$M^{N}\vec 1=\Bigl(
(d-1)2^{N-1}\!+\!1, \  (d-1)2^{N-2}\!+\!1, \  \ldots , \   (d-1)2^0\!+\!1, \! \underbrace{1, \ldots, 1}_{(d-N)-\text{times}}\!\!\Bigr)^T
\,.$$ 

Using Lemma~\ref{lem:s_n_simplified}, we have $s_N={\vec 1\hspace{0.01cm}}^T M^{N}\vec 1=(d-1)2^N+1$ for $1\leq N \leq d-1$.
 
The characteristic polynomial of $M$ equals $x^d-x^{d-1}-\cdots -x-1$. Therefore, by the Cayley-Hamilton theorem, we have $M^d=M^{d-1}+\dots +M+I$. Consequently,
$M^N=M^{N-1}+\dots+M^{N-d}$ for $N\geq d$. Using Lemma~\ref{lem:s_n_simplified}, multiplying both sides of the equation by ${\vec 1\hspace{0.01cm}}^T$ from the left and $\vec 1$ from the right gives us the announced recurrence relation for $(s_N)$.
\end{proof}

Using the standard solution procedure of the recurrence relations (see Appendix), we get 
\begin{equation}\label{eq:explicit}
    s_N=\sum_{k=1}^d c_k t_k^N, \quad \text{where $c_k=\frac{(d-1)t_k}{(d+1)t_k-2d}$}\,.
\end{equation}

In the sequel, we will use some properties of the polynomial $p(x)=x^d-x^{d-1}-\cdots-x-1$. 
  As proven in \cite{Ba1951},  it is the minimal polynomial of a Pisot number, i.e., one root, say   $t_1=t>1$  and all other roots $ t_2, \dots, t_d$ are of modulus strictly smaller than 1.

\begin{description}
\item[Property 1]:  \ \ $2-\frac{1}{2^{d-1}}<t<2-\frac{1}{2^d}$ (see~\cite{BPT18});
\item[Property 2]:  \ \  $t_k^{d}=\frac{1}{2-t_k}$, \ as  $0=(t_k-1)p(t_k) =t_k^{d+1}-2t_k^d+1$.
\end{description}

\medskip

\begin{theorem}\label{thm:EandE*} Let $\uu_d$ be the $d$-bonacci sequence. Then 
$$E(\uu_d)=E^*(\uu_d)=2+\frac{1}{t-1},$$ where $t>1$ is the unique positive root of the polynomial $x^d-x^{d-1}-\cdots -x-1$.
\end{theorem}
\begin{proof} Since $t_1 = t >1$ and all other roots $t_k$ of the polynomial $p(x)$ are in modulus smaller than 1,  Equation~\eqref{eq:explicit} implies 
\begin{equation}\label{eq:asymptotika}s_N = c_1t^N + o(1).\end{equation}  
Using Theorem~\ref{thm:formulaE*andE},  we have 
$$E^*(\uu_d)=1+\lim_{N \to \infty} \frac{s_N}{s_N-s_{N-1}}=1+\frac{t}{t-1}=2+\frac{1}{t-1}\,.$$

As  $E^*(\uu_d) \leq E(\uu_d)$ by definition,  it suffices to show that $E(\uu_d) \leq 2+\frac{1}{t-1}$ in order to complete the proof of Theorem~\ref{thm:EandE*}. That is, according to Theorem~\ref{thm:formulaE*andE}, to prove  for every $N \geq 1$
$$1+\frac{s_N-d}{s_N-s_{N-1}} \leq 1+\frac{t}{t-1}\,,  \ \text{ or equivalently, } \ \  t s_{N-1}-s_N \leq (t-1)d.$$
For $1 \leq N \leq d$, it suffices to use the initial conditions for $s_N$ from Lemma~\ref{lem:recurrence_s_n} and the fact that $t<2$ to verify the inequality.

For $N \geq d$, we  exploit that $2>t=t_1>1$ and the root $t_k$ is in modulus smaller than  $1$ for $k \in \{2, \ldots, d\}$.   Hence,  
$$\Bigl|\frac{d-1}{(d+1)t_k-2d}\Bigr| \leq \frac{d-1}{2d - (d+1)} = 1\qquad \text{and} \qquad {|t-t_k|}<  {|2-t_k|}.$$ 
Replacing  $s_N$ by the expression~\eqref{eq:explicit}  and using  Property  2 we get 
$$t s_{N-1}-s_N=\sum_{k=2}^d \frac{(d-1)(t-t_k)t_k^{N}}{(d+1)t_k-2d} \leq \sum_{k=2}^d |(2-t_k)t_k^N| =\sum_{k=2}^d {|t_k|}^{N-d}$$
Therefore, for $N\geq d$, 
$$ t s_{N-1}-s_N < \sum_{k=2}^d 1= d-1 < (t-1)d\,,$$
where the last estimate comes from  Property 1. This completes the proof.

\end{proof}

\begin{remark}\label{rem:RT}
For a better idea, let us inspect more the value $E(\uu_d)$.
To emphasize the size of the alphabet, let us write $t(d)$ for the unique positive root of $x^d-x^{d-1}-\cdots-x-1$.
By Property 1, we have $$3+\frac{1}{2^d-1}<E(\uu_d)=2+\frac{1}{t(d)-1}<3+\frac{1}{2^{d-1}-1} <E(\uu_{d-1})\,.$$
Thus, $E(\uu_d) = E^*(\uu_d)$ is strictly decreasing when the alphabet size increases and $\lim\limits_{d \to \infty} E(\uu_d)=3$.
For illustration, we computed the approximate values of $E(\uu_d)$ for $d\in\{2,3,\dots,7\}$:
$$\begin{array}{c|c|c|c|c|c|c}
    d &  2 & 3 & 4 & 5 & 6 & 7 \\ \hline
    t(d) & 1.618 &  1.839 & 1.928 & 1.966 & 1.984 & 1.992\\ \hline
    E(\uu_d) 
    & 3.618 &  3.191 & 3.078 & 3.035 & 3.017 & 3.008\\ \hline
\end{array}$$
\end{remark}

\section{Maximal length of the $N$-th bispecial factor in AR sequences}
In this section, for a fixed alphabet size $d\geq 2$ and a fixed $N \in \N$, we will determine an AR sequence $\uu$ with the longest $N$-th bispecial factor $b_N$.  We will show that it is exactly the $d$-bonacci sequence that has the longest bispecial factors. For this purpose, we need several technical lemmata.

According to Lemma \ref{lem:BS_shortest_retwords} and notation from Theorem \ref{thm:formulaE*andE}, it holds for the length $|b_N|$  of the $N$-th bispecial factor that $s_N = (d-1)|b_N| +d$. Therefore, it suffices to find the maximum  of the set 
\begin{equation}\label{eq:mnozinaS} \mathcal{S}(d,N) =\Bigl\{\vec 1^{\,T}\, M_{h_1} M_{h_2}\cdots M_{h_N}\,\vec 1\ :  h_1, h_2, \ldots, h_N \in \{0,1,\ldots, d-1\} \Bigr\}\,.\end{equation}

\begin{lemma}\label{le:oPermutaci} Let $\pi$ be a permutation on the set $\{0,1,\ldots, d-1\}$  and $P \in \N^{d\times d}$ be the permutation matrix  corresponding to $\pi$. Then 
$P^T M_kP = M_{\pi(k)}$ for every
$k\in \{0,1,\ldots, d-1\}$.    
\end{lemma}
\begin{proof} Every permutation is a composition of transpositions. It thus suffices to prove the statement for any transposition $\pi$. 

Let $i,j \in \{0,1,\ldots, d-1\}, i\neq j$, and let $\pi$ be the transposition exchanging  $i\leftrightarrow j$. A straightforward computation gives
$P^TM_kP=M_k$, for $k \neq i,j$, $P^TM_jP=M_i$ and  $P^TM_iP=M_j$. Hence $P^TM_kP = M_{\pi(k)}$ for every $k$. 
\end{proof}

In the sequel we will for  vectors $\vec u, \vec v \in \mathbb{R}^d$ write $\vec u \leq \vec v$ if  the $i^{th}$ component of $\vec u$ is smaller than or equal to the $i^{th}$ component of $\vec v$, for every $i \in \{0,1,\dots, d-1\}$. We say that a vector $\vec v\in \mathbb{R}^d$ is {\em positive} if each of its components is a~positive number. We will also use the following property of the matrix $M_k$ of the morphism  $\varphi_{\tt k}$
$$M_k\vec e_{k} =  \vec e_{k} \quad \text{and}\quad  M_k\vec e_{j} =  \vec e_{k}  + \vec e_j\quad  \text{for\ } j \neq k.$$

\begin{lemma}\label{le:prvnisloupec} Let $k \in \{1,2,\dots, d-1\} $. Denote $A=M_0M_1\cdots M_{k-1}$.  Then $A \vec e_0 \leq A\vec e_j$ for every $j \geq k$. Moreover, the inequality is strict in at least one component of the vectors.   
\end{lemma}
\begin{proof} By induction on $k$.  

$k=1:$   \ \ In this case $A = M_0$. Thus  $A\vec e_0 = \vec e_0$ and $A\vec e_j = \vec e_0 +\vec e_j$  for every $j\geq 1$.  Hence for  each $j \geq 1$  we have $A \vec e_0 \leq A\vec e_j$ and  the inequality  is strict in the $j^{th}$ component. 

$1\leq k < d-1: $ \  Denote  $\tilde{A} = M_0M_1\cdots M_{k-1}M_{k}$. Then $\tilde{A} \vec e_j =  A M_{k} \vec e_j = A(\vec e_j+ \vec e_{k} ) =A\vec e_j+ A\vec e_{k}  $ for every $j \neq k$.  In particular, 
$\tilde{A} \vec e_0 = A \vec e_0 + A \vec e_{k}$. 
By the induction hypothesis, $A \vec e_0 \leq A\vec e_j$  for every $j\geq k$ and thus $\tilde{A} \vec e_0 \leq \tilde{A} \vec e_j $ for  every $j\geq k+1$.     
\end{proof}

  \begin{lemma}\label{le:ostraNerovnost} Let $k \in \{1,2,\dots, d-1\}$. Then there exists a permutation matrix $R$ such that the following inequality   
  $$
  \vec x^{\, T} \,M_0M_1 \cdots M_{k-1} M_0\,R \, \vec z \  <  \     \vec x^{\, T} \,M_0M_1 \cdots M_{k-1} M_k \,\vec z 
  $$
holds true for every pair of positive vectors $\vec x$ and $\vec z$  from  $\mathbb{R}^d$. 
  \end{lemma}

\begin{proof} Set $R$ equal to the matrix of the transposition $\rho$ that exchanges $0\leftrightarrow k$. 
We will show that the value  $v:= \vec x^{\, T}\,M_0M_1 \cdots M_{k-1}(M_k - M_0R)\,\vec z$ is positive. Denote $\vec z =(z_1, z_2,\ldots, z_d)^T$ and $A:=M_0M_1\cdots M_{k-1}$.  
Simple calculations give
$$
(M_k - M_0\,R)\, \vec z = \xi (-\vec e_0+\vec e_k), \quad \text{where}\ \xi = z_2+z_3+\cdots + z_d >0\,. 
$$

Therefore $v = \xi\vec x^{\,T}\,A(-\vec e_0 + \vec e_k)=\xi\vec x^{\,T}\,(-A\vec e_0 + A\vec e_k) $. By Lemma \ref{le:prvnisloupec}, the vector $-A\vec e_0 + A\vec e_k$ has non-negative components and at least one of its components is positive.  As $\vec x$ has all components positive and $\xi >0$, their product $v$ is a positive value. 

\end{proof}

\begin{proposition} Let $N \in \N$ and $i_1, i_2, \ldots, i_N \in \{0,1,\ldots, d-1\}$ be such that  
$$\vec 1^{\,T} M_{i_1} M_{i_2} \cdots M_{i_N}\, \vec 1 = \max \mathcal{S}(d,N), $$
where $\mathcal{S}(d,N)$ is defined by \eqref{eq:mnozinaS}. Then any $d$ consecutive indices in the list $i_1,i_2, \ldots, i_N$ are mutually distinct.   
\end{proposition}

\begin{proof}  By contradiction, find the smallest $k \in \{1,2,\dots, d-1\}$ and $m \geq 1$ such that in the list $i_{m},i_{m+1}, \ldots, i_{m+k}$  two indices coincide. Divide  the product of matrices into three parts 

$X := M_{i_1} M_{i_2} \cdots M_{i_{m-1}}$, 

$Y : = M_{i_m}M_{i_{m+1}} \cdots M_{i_{m+k-1}} M_{i_{m+k}}$,  

$Z:=M_{i_{m+k+1}} M_{i_{m+k+2}}\cdots M_{i_N}$.

\medskip

As we take the smallest $k$, necessarily $i_{m}, i_{m+1}, \ldots, i_{m+k-1}$ are mutually distinct and $i_{m+k} = i_{m}$. Consider a~permutation $\pi$  on the set $\{0,1,\ldots, d-1\}$ such that $\pi(i_{m +j}) = j$ for $j \in \{0, 1, \ldots, k-1\}$.  Let $P$ be the matrix of the permutation $\pi$. By Lemma \ref{le:oPermutaci}, 
$$P^TYP=P^T M_{i_m}M_{i_{m+1}} \cdots M_{i_{m+k-1}} M_{i_{m+k}} \,P =  M_0M_1 \cdots M_{k-1}M_0\,.$$
Let $R$ be the permutation matrix from Lemma \ref{le:ostraNerovnost}.  Since $\vec 1^{\, T} P^T = \vec 1^{\,T}$, $PR \vec 1 = \vec 1$,  $P^TP =R^TR = I$, we can write
$$\max\mathcal{S}(d,N) = \vec 1^{\, T}\, X\,Y\,Z \, \vec 1 = \vec 1^{\, T} \,  (P^T X P) \, (P^T Y P\,R)\, (R^T P^T Z P \, R )\,\vec 1\,. $$
Now, we set $\vec x^{\,T} = \vec 1^{\, T} \, P^T X P $  and $\vec z = R^TP^T Z PR \,\vec 1$.
Using this notation, we have  
     $$\max\mathcal{S}(d,N) = \vec x^{\,T } P^TY P R\vec z = \vec x^{\,T }  M_0M_1 \cdots M_{k-1}M_0 R \vec z\,.$$ 
According to Lemma \ref{le:ostraNerovnost}, we obtain
  $$\max\mathcal{S}(d,N)  <  \vec x^{\,T }  M_0M_1 \cdots M_{k-1}M_k  \vec z\,. $$
To derive a contradiction, it suffices to prove that the number $\vec x^{\,T }  M_0M_1 \cdots M_{k-1}M_k  \vec z$ \  is an element of the set $\mathcal{S}(d,N) $.  

Let $\rho$ be the permutation corresponding to the matrix $PR$. Then 

$\vec x^{\,T}=  \vec 1 ^{\, T} P^TXP =\vec 1 ^{\, T}\,M_{\pi(i_1)} M_{\pi(i_2)} \cdots M_{\pi(i_{m-1})}$, 

$\vec z = R^TP^T ZPR \vec 1 = M_{\rho(i_{m+k+1})} \cdots M_{\rho(i_N)}\, \vec 1 $.

Hence $ \vec x^{\,T }  M_0M_1 \cdots M_{k-1}M_k  \vec z $  is an element of the set $ \mathcal{S}(d,N)$  corresponding to the $N$ indices:  ${\pi(i_1)}, {\pi(i_2)},  \ldots, {\pi(i_{m-1})}, 0,1, \ldots, k, {\rho(i_{m+k+1})}, \ldots, {\rho(i_N)} $. 
\end{proof}

The previous proposition states that the maximum of the set $\mathcal{S}(d,N)$ is attained for the AR sequence having the directive sequence $\Delta=(\varphi_{\tt 0}\varphi_{\tt 1}\cdots \varphi_{\tt d-1})^{\omega}$, i.e., for the $d$-bonacci sequence.  
Combining it with \eqref{eq:asymptotika} we have 
\begin{corollary}\label{coro:dBonacciMaxim} Let $d \in \N, d\geq 2$. There exists positive $c_1 >0$  and a sequence $(o_N)_{N\in \N}$ with $\lim\limits_{N\to \infty} o_N = 0$ such that 
$$\max  \mathcal{S}(d,N) = c_1t^N +  o_N.  $$
\end{corollary}

\section{Repetition threshold of $d$-ary episturmian sequences}
In this section we state  our  main result as Theorem \ref{th:main}.  The key ingredient  of its  proof  is the  following proposition.  
\medskip

\begin{proposition}\label{pro:d-bonacciBest} Let $\uu$ be a $d$-ary AR sequence. Then $E^*(\uu) \geq E^*(\uu_d)$.    
    
\end{proposition}
\begin{proof}
 According to Theorem~\ref{thm:formulaE*andE}, we have for any $d$-ary AR sequence $\uu$
$$\begin{array}{rcl}
E^*(\uu)&=&1+\limsup\tfrac{s_N-s_{N-1}+s_{N-1}}{s_N-s_{N-1}}\\
&=&2+\limsup\frac{1}{\tfrac{s_N}{s_{N-1}}-1}\\
&=&2+\frac{1}{\liminf\tfrac{s_N}{s_{N-1}}-1}\,.
\end{array}$$
Now we use  the Cauchy's theorem on limits: Let $(a_N)$ be a~positive sequence, then 
$$\liminf \frac{a_N}{a_{N-1}}\leq  \liminf \sqrt[N]{a_N}\leq \limsup \sqrt[N]{a_N}  \leq \limsup \frac{a_N}{a_{N-1}} .$$
Using Corollary \ref{coro:dBonacciMaxim} we deduce 
$$ \liminf\frac{s_N}{s_{N-1}}\leq  \limsup \sqrt[N]{s_N} \leq   \limsup \sqrt[N]{\max \mathcal{S}(d,N)} = t.$$
Altogether, 
$$E^*(\uu) = 2+\frac{1}{\liminf\tfrac{s_N}{s_{N-1}}-1} \geq 2+ \frac{1}{t-1} = E^*(\uu_d)\,, $$  where the last equality follows by Theorem \ref{thm:EandE*}.     
\end{proof}

\begin{theorem}\label{th:main}
The repetition threshold  and the asymptotic repetition threshold of the class of $d$-ary episturmian sequences are  equal to $2+\frac{1}{t-1}$, where $t>1$ is the unique positive root of the polynomial $x^d-x^{d-1}-\cdots -x-1$.
It is reached by the $d$-bonacci sequence.
\end{theorem}

\begin{proof}
By the definitions of critical exponent and asymptotic critical exponent,   $E(\uu)\geq E^*(\uu)$ for any sequence $\uu$.    By Theorem \ref{thm:EandE*}, $E^*(\uu_d) = E(\uu_d) = 2+\frac{1}{t-1}$. Consequently, to  prove the theorem we have to show $E^*(\uu)\geq E^*(\uu_d)$ for every $d$-ary episturmian sequence $\uu$. Let us discuss three cases: 
\begin{itemize}
    \item If $\uu$ is an  AR sequence, then  $E^*(\uu) \geq E^*(\uu_d)$ by  Proposition \ref{pro:d-bonacciBest}.    
    \item If $\uu$ is periodic, then obviously $E^*(\uu)=\infty > E^*(\uu_d)$.
    \item 
If $\uu$ is aperiodic and not an AR sequence, then $\uu = \psi(\vvv)$ for some non-erasing morphism $\psi$ and $d'$-ary AR sequence $\vvv$ with  $2\leq d '< d$.  By Lemma \ref{lem:morphismE}, Proposition \ref{pro:d-bonacciBest} and Remark \ref{rem:RT}, we have
$E^*(\uu) \geq E^*(\vvv) \geq  E^*(\uu_{d'}) >  E^*(\uu_d)$.  
\end{itemize} 

\end{proof}

\section{Comments}

In the class $C_d$ of $d$-ary episturmian sequences, the  repetition threshold and the asymptotic repetition threshold are attained on the $d$-bonacci sequence. 
For $d=2$, Carpi and de Luca showed \cite{CaDeLu00} that the repetition threshold $2+\frac{1+\sqrt{5}}{2}$ is reached on 6 sequences in the class $C_2$, one of them is of course the Fibonacci sequence. Moreover, there are two sequences with the critical exponent equal to $11/3$ and all other ones have already the critical exponent $\geq 4$.

It would be interesting to find an analogy of their result even for $d\geq 3$. In addition, it is not clear how the structure of the set $\{E(\uu): \uu \in C_d \}$ looks like. In particular, what is the second, third, etc. smallest element of this set, resp. whether this set has accumulation points and if yes, what is the smallest one among them.

For the class of binary sequences rich in palindromes, where $RT(C_2) = 2+ \frac{\sqrt{2}}{2} \sim 2.707$, this problem is also partially solved: Currie, Mol and Rampersad~\cite{CuMoRa2020} described all rich sequences with the critical exponent smaller than $2.8$.

This is in contrast to the results for general sequences. As shown in \cite{CuRa08} for $d=2$ and in \cite{Vas11} for $d=3$, each $\alpha \geq RT(d)$ is the critical exponent of a $d$-ary sequence.

As we have already mentioned, $d$-ary episturmian sequences are rich in palindromes, in particular, the $d$-bonacci sequence is rich. Hence, on one hand, for the class $C_d$ of sequences rich in palindromes we can conclude $RT(C_d) \leq  E(\uu_d)$. This bound on the repetition threshold was mentioned already by Vesti in \cite{Vesti2017}. Due to the ambiguity in terminology, Vesti swapped the value $E^*(\uu_d)$ (deduced by Glen and Justin) for the value $E(\uu_d)$. Since here we have proved that these two values coincide, Vesti's bound  remains valid.  On the other hand, we have the lower bound $RT(C_d)\geq RT^*(C_d)\geq 2$ since every rich sequence contains infinitely many overlapping factors, see~\cite{PeSt}.

\section*{Acknowledgements}

The authors  acknowledge financial support of M\v SMT by  founded project  \\  CZ.02.1.01/0.0/0.0/16\_019/0000778.

\section{Appendix}
\begin{lemma}\label{le:appendix} Let $d\geq 2$ and $(s_N)$ be the sequence given by the $d$-bonacci recurrence relation
$$s_N=s_{N-1}+s_{N-2}+\dots +s_{N-d} \quad \text{for $N\geq d$}$$
with the initial values $$s_N=(d-1)2^N+1 \quad \text{for} \ \ N \in \{0,1,\dots, d-1\}\,. $$
Denote  $t_1, t_2, \dots, t_d$   zeros  of the polynomial  $p(x)=x^d-x^{d-1}-\dots -x-1$. Then    $$s_N=\sum_{k=1}^d c_k t_k^N, \quad \text{ where} \ \  c_k=\frac{(d-1)t_k}{(d+1)t_k-2d}\,.$$
\end{lemma}
\begin{proof}
Since all roots of the polynomial $p(x)$ are mutually distinct, it is well-known from the theory of linear recurrence relations that there exist constants $c_1, c_2, \cdots, c_d$, depending on the initial conditions, such that $s_N=\sum_{k=1}^d c_k t_k^N$.

To determine the constants $c_k$, we have to solve the following system of linear algebraic equations for $c_1, \dots, c_d$:
$$\left(\begin{array}{cccc}
     1&1&\dots &1  \\
     t_1&t_2&\dots &t_d\\
     t_1^2 & t_2^2 & \dots & t_d^2\\
     \vdots & \vdots & \ddots & \vdots \\
     t_1^{d-1} & t_2^{d-1} & \dots & t_d^{d-1}
\end{array}\right)\left(\begin{array}{c} c_1 \\ c_2 \\ c_3 \\ \vdots \\ c_d \end{array}\right)=(d-1)\left(\begin{array}{c} 1 \\ 2^1 \\ 2^2 \\ \vdots \\ 2^{d-1} \end{array}\right)+\left(\begin{array}{c} 1 \\ 1 \\ 1 \\ \vdots \\ 1 \end{array}\right)\,.$$

Using the Cramer's rule we get 
$$c_1 = \frac{(d-1)V(2, t_2,t_3, 
\ldots, t_d) + V(1,t_2,t_3, \ldots, t_d)}{V(t_1,t_2, t_3, \ldots, t_d)}\,, $$
where $V(\alpha_1, \alpha_2, 
\ldots, \alpha_d)$ denotes  the Vandermonde determinant of complex numbers $\alpha_1, \alpha_2, \ldots, \alpha_d$. It is well known that 
$$ V(\alpha_1, \alpha_2, 
\ldots, \alpha_d) = \prod_{d\geq k > \ell \geq 1} (\alpha_{k} - \alpha_\ell)\,.
$$
Hence, $$\begin{array}{rcl}
c_1&=&\cfrac{\displaystyle(d-1)\prod_{j\geq 2}(t_j-2) + \prod_{j\geq 2}(t_j-1)}{\displaystyle\prod_{j\geq 2}(t_j-t_1)}\,.
\end{array}$$
To simplify the products in the above fraction, we use some useful and easy to check properties of the characteristic polynomial $p(x)$, the polynomial $q(x) := (x-1)p(x)$ and their derivatives:     
\begin{enumerate}
\item $p(x)=\prod_{j=1}^d (x-t_j)$;
\item $q(x)=x^{d+1}-2x^d+1$; 
\item $q'(x) = p(x) + (x-1)p'(x) = (d+1) x^d - 2dx^{d-1}$;
\item  $t_k^{-d}={2-t_k}$ as $t_k$ is a root of $q(x)$ for every $k=1, 2, \ldots, d$;  
\item $\prod_{j=1}^d (1-t_j) =p(1)= 1-d$;   
\item $p'(t_1)=\prod_{j\geq 2}(t_1-t_j)$; 
\item $q'(t_1) = (t_1-1)p'(t_1)=(d+1)t_1^d-2dt_1^{d-1}$.
\end{enumerate}
Applying the above properties gives us

$\prod\limits_{j\geq 2}(2-t_j) = \prod\limits_{j\geq 2}t_j^{-d} =  t_1^{\,d} \Bigl(\prod\limits_{j\geq 1}t_j\Bigr)^{-d} =t_1^{\,d}\, {\bigl((-1)^d p(0)\bigr)}^{-d} = t_1^{\,d}\,; $ 

\medskip

 $\prod\limits_{j\geq 2}(1-t_j)= (1-t_1)^{-1}\prod\limits_{j\geq 1}(1-t_j) = (d-1)(t_1-1)^{-1}$\,;

\medskip

  $\prod\limits_{j\geq 2}(t_1-t_j)= p'(t_1) = (t_1-1)^{-1}\, q'(t_1) = {t_1^{d-1}}(t_1-1)^{-1}\, \bigl((d+1)t_1 - 2d \bigr)$\,.

\medskip

\noindent Inserting the simplified expressions to the formula for $c_1$, we obtain $$c_1 = (d-1)\frac{(t_1-1)t_1^d +1 }{{t_1^{d-1}}\, \bigl((d+1)t_1 - 2d \bigr)}=(d-1)\frac{t_1^d }{{t_1^{d-1}}\, \bigl((d+1)t_1 - 2d \bigr)}=\frac{(d-1)t_1 }{(d+1)t_1 - 2d }\,.$$
 
Since the role of coefficients $c_k$ is symmetric with respect to the roots $t_k$, we get $c_k$ from $c_1$ just by replacing $t_1$ by $t_k$. 
\end{proof}

\end{document}